\documentclass[final]{siamart0216}
\pdfoutput=1
\usepackage{graphicx}
\usepackage{amssymb}

\usepackage{multirow}
\usepackage[labelformat=simple]{subcaption}
\usepackage{color}

\newtheorem{remark}{Remark}[section]

\title{Analysis of rippling in incommensurate one-dimensional coupled chains}

\author{Paul Cazeaux%
\thanks{School of Mathematics, University of Minnesota, Minneapolis, MN 55455 (\email{pcazeaux@umn.edu}, \email{luskin@umn.edu})}
\and
Mitchell Luskin%
\footnotemark[1]
\and 
Ellad B. Tadmor%
\thanks{Department of Aerospace Engineering and Mechanics, University of Minnesota, Minneapolis, MN 55455 (\email{tadmor@umn.edu})}
}
\date{\today}

\numberwithin{equation}{section}
\begin{document}
\maketitle
\begin{abstract}
Graphene and other recently developed 2D materials exhibit exceptionally strong in-plane stiffness. Relaxation of few-layer structures, either free-standing or on slightly mismatched substrates occurs mostly through out-of-plane bending and the creation of large-scale ripples.
In this work, we present a novel double chain model, where we allow relaxation to occur by {\it bending} of the incommensurate coupled system of chains.  As we will see, this model can be seen as a new application of the well-known Frenkel-Kontorova model for a one-dimensional atomic chain lying in a periodic potential. We focus in particular on modeling and analyzing ripples occurring in ground state configurations, as well as their numerical simulation.

%approximating the incommensurate system by commensurate systems.
\end{abstract}



\section*{Introduction}
Two-dimensional layered crystals such as graphene, nature's thin-nest elastic material, exhibit exceptionally strong in-plane stiffness but are also highly flexible with extremely low flexural rigidity compared to conventional membranes.
Mono- and few-layer graphene structures display spontaneous creation of long-range ripples which have been experimentally observed either in suspended conditions~\cite{meyer2007structure} or supported on slightly mismatched substrates~\cite{de2008periodically,singh2010atomic,woods2014commensurate}.
Ripples in graphene is a multiscale phenomenon that is expected to influence its mechanical and electronic properties~\cite{neto2009electronic}, and inducing periodic ripples in a controlled fashion may open new perspectives, e.g., for graphene-based devices~\cite{bao2009controlled,miranda2009graphene}.

Various mechanisms have been described as responsible for the appearance of ripples in two-dimensional materials. Spontaneous ripples in suspended graphene monolayers have been attributed to thermal fluctuations~\cite{fasolino2007intrinsic} or stress~\cite{bao2009controlled}. Graphene monolayers supported by an almost-commensurate substrate relax to form a periodic vertical corrugation pattern~\cite{de2008periodically,ni2012quasi,woods2014commensurate,wang2016entropic}. Sharp out-of-plane folds called ripplocations have been predicted in van der Waals (vdW) homostructures such as MoS$_2$ multilayers, similar to line defects~\cite{kushima2015ripplocations}. In a ripplocation, an additional local line of atoms is inserted in one layer leading to the formation of an out-of-plane wrinkle, while the other layer remains flat.

In this paper, we will study a different mechanism yet, where the spontaneous atomic-scale relaxation of free-standing systems of incommensurate vdW bilayers leads to a \textit{simultaneous} long-range rippling of the bilayer system~\cite{IliaPaper}. Note that vdW multi-layer structures tend to form naturally incommensurate stackings, either due to a relative rotation of the crystalline orientations or to a natural mismatch between the respective lattice constants.
To model mathematically this multiscale phenomenon, we present a new double chain model, where we allow relaxation to occur by {\it bending} of the incommensurate coupled system of chains, as shown in \cref{fig:SimulatedRipples}. We focus in particular on modeling and analyzing rigorously ripples occurring in ground state configurations, as well as their numerical simulation. Our model was motivated in part by the related continuum model of Nikiforov and Tadmor ~\cite{IliaPaper}.

The model~\cref{eq:potentialenergy1,eq:potentialenergy2,eq:siteenergy} developed in this work can be seen as a new application of the well-known Frenkel-Kontorova model for a one-dimensional atomic chain lying in an incommensurate periodic potential~\cite{aubry1983devil,AubryLeDaeron1983,de1984critical,BraunKivshar2004frenkel}. The role of the periodic potential is here assumed by the local interlayer contribution to the interaction energy, which is influenced by the local atom stacking configuration of the bilayer. Indeed, staggered atoms (maximizing the number of neighbors) minimize the energy compared to locally aligned configurations as seen on \cref{fig:LocalConfigurations} (see \Cref{sec:motivating_numerical_example} for details). We will show how, by taking on some bending strain, the bilayer system is able to significantly reduce its energy by maximizing the area of staggered configurations.

\begin{figure}[t]
\centering
\includegraphics[width=1\textwidth]{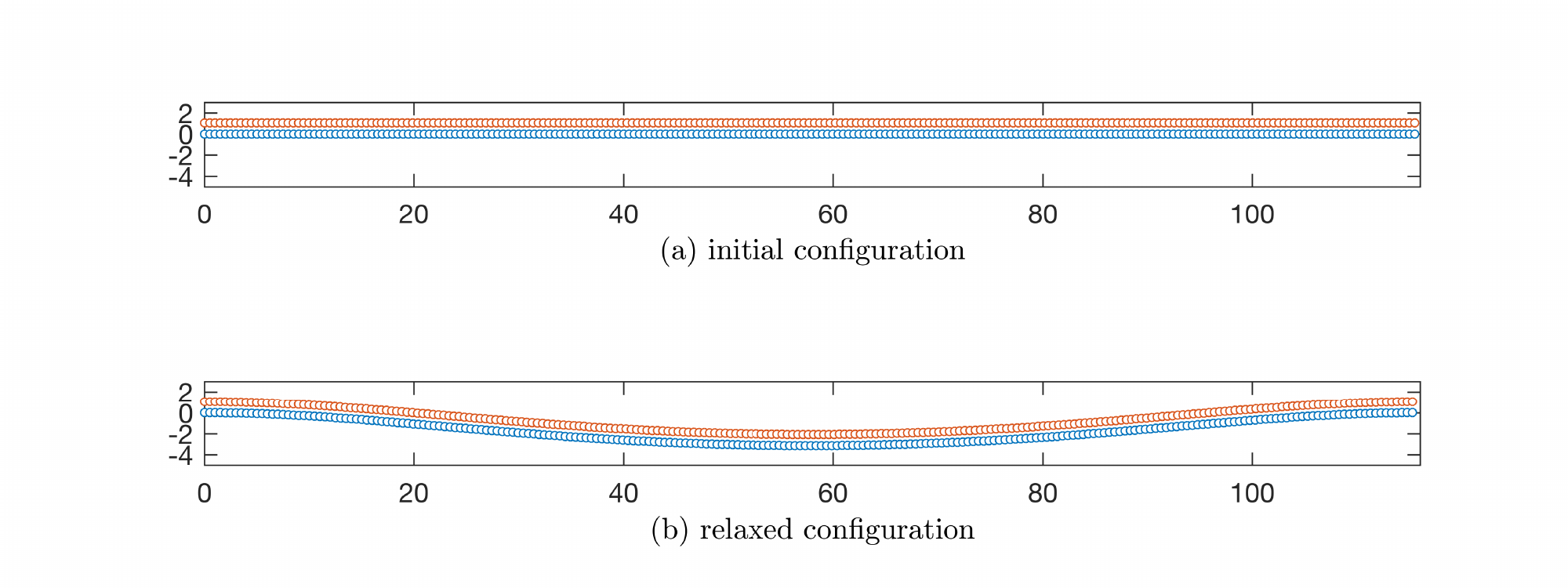}
\caption{Numerical relaxation of coupled chains by the creation of ripples.}
\label{fig:SimulatedRipples}
\end{figure}

Let us note that it is not possible for perfectly inextensible 2D sheets to be curved in more than one direction. As a result, a reasonable 2-dimensional model for the rippling of incommensurate layers should account for the slight extensibility of the individual layers. Thus, many simplifications used in this work, such as Eq.~\cref{eq:curvature}, cannot be used outside of our one-dimensional chain toy model.

This paper is organized as follows. First, we motivate our approach by presenting a simple numerical atomistic simulation to illustrate the phenomenon and motivate our model (see also the related example in~\cite{IliaPaper}). In \Cref{sec:DerivationModel}, we introduce a mathematical model for the rippling mechanism, and we discuss the various assumptions that allow our explicit construction. In \Cref{sec:GroundStateAnalysis}, we recall the classical analytical results on the Frenkel-Kontorova model due to Aubry and Le Daeron~\cite{AubryLeDaeron1983} that allow us to identify the ground state configurations. In \Cref{sec:ApproximationPeriodic}, we show how the popular supercell method of approximating the incommensurate system by periodic approximations can be rigorously justified. Finally, we present in \Cref{sec:NumericalExamples} numerical computations that illustrate and support our analytical results.

\begin{figure}[t]
\centering
\begin{subfigure}{.49\textwidth}\centering
	\includegraphics[width=.8\textwidth]{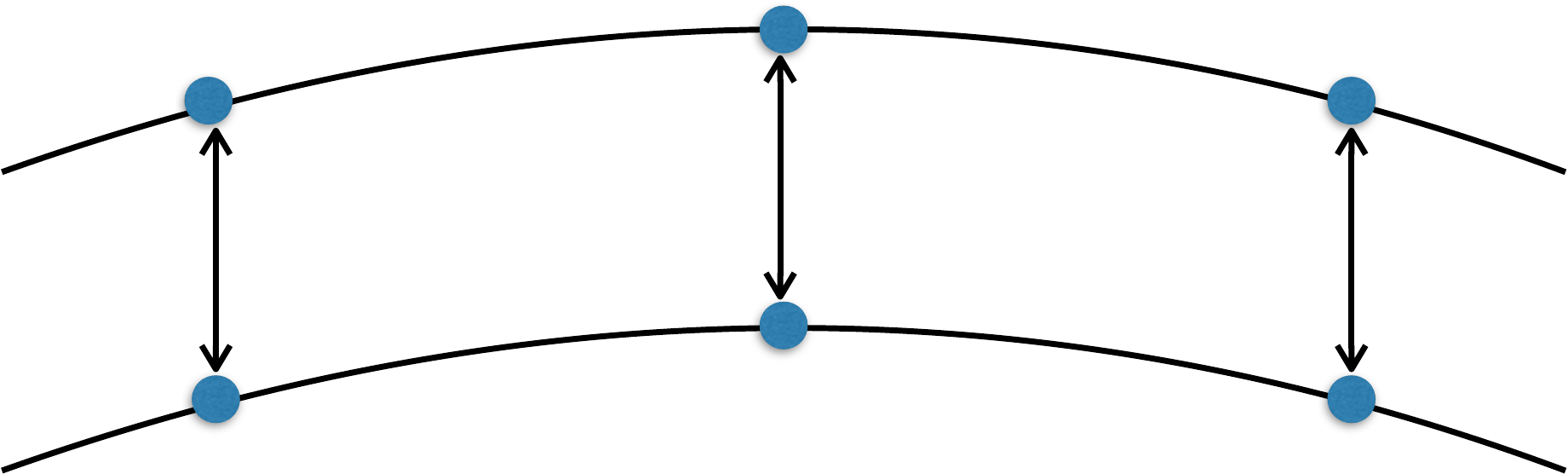}
	\caption{Aligned atomic configuration}\label{fig:LocalConfigurations:a}
\end{subfigure}
\begin{subfigure}{.49\textwidth}\centering
	\includegraphics[width=.8\textwidth]{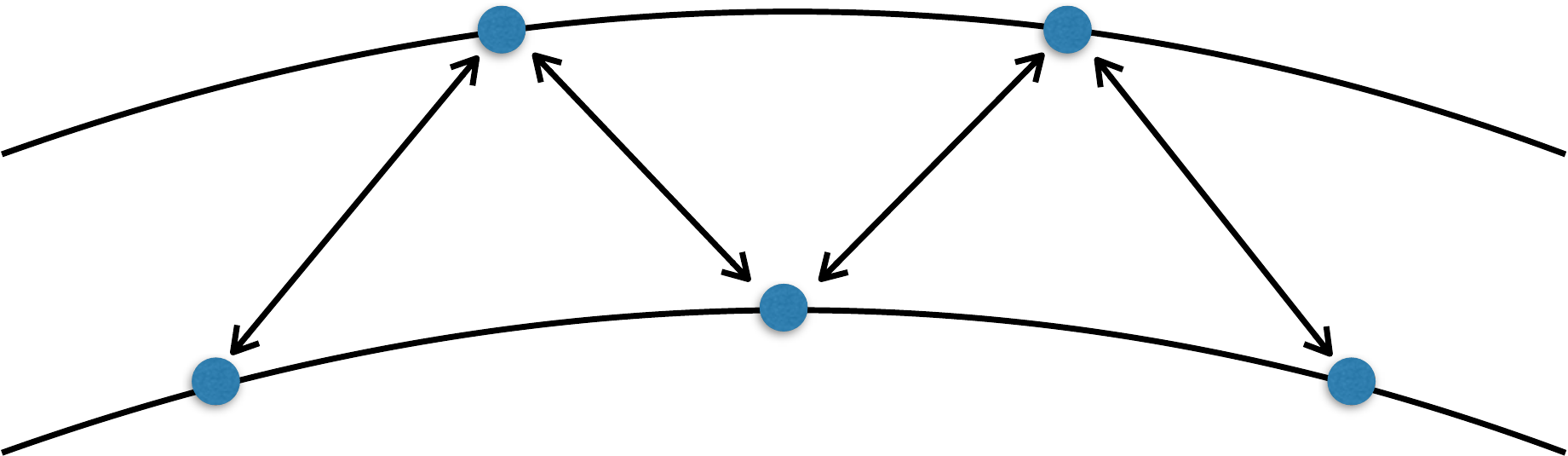}
	\caption{Staggered atomic configuration}\label{fig:LocalConfigurations:b}
\end{subfigure}
\caption{Local atomic configurations: aligned situation~\subref{fig:LocalConfigurations:a} leads to a higher energy contribution than the staggered configuration~\subref{fig:LocalConfigurations:b}.}\label{fig:LocalConfigurations}
\end{figure}

\section{Motivating numerical example} % (fold)
\label{sec:motivating_numerical_example}

We present first a simple numerical atomistic simulation to illustrate the phenomenon. Two coupled one-dimensional chains are fitted to represent a single zigzag row of dimers in graphene~\cite{IliaPaper}. Within each chain, harmonic bonds model the interaction between nearest neighbors, and a harmonic angular spring is centered on each atom involving its two nearest neighbors. Between the two chains, atoms interact with a long-range Lennard-Jones potential%
\footnote{
It is common to use the Lennard-Jones potential to model the weak interlayer van der Waals interactions in 2D layered materials. It can describe the overall cohesion of the material quite well, however, it is also too smooth to correctly describe energy variations between different stacking configurations. Truly realistic computations should thus use a more accurate potential, such as the graphitic disregistry-dependent potential developed by Kolmogorov and Crespi~\cite{kolmogorov2000smoothest}. For simplicity, and without loss of generality, we will use in this paper the Lennard-Jones potential, acknowledging that this actually leads to underestimating the formation of ripples.
}.%
Throughout this section, we work in Lennard-Jones units: $\epsilon = 2.39$ meV is the unit of energy and $\sigma = 3.41$ {\AA} is the unit of length for graphene-graphene interaction, see~\cite{girifalco2000carbon}.

To simulate an incommensurate system, we create a periodic supercell by choosing the equilibrium length to be slightly different in each chain: we take $N_1 = 232$ and $N_2 = 233$ atoms respectively in each chain for a total length of $116\ \sigma$. Periodic boundary conditions are applied at the ends. Numerical parameters are fitted to match a zigzag row of dimers in graphene~\cite{kudin2001c,IliaPaper}, leading to the parameterization given in \cref{tab:NumParametersAtomistic}.
\begin{table}[h]
	\caption{Numerical parameters for the atomistic simulations.}
	\label{tab:NumParametersAtomistic}
	\centering

	\begin{tabular}{l|ccc}
	\hline

	\hline
	\multirow{2}{4.9em}{\textbf{Graphene}} & \textbf{Row width} & \textbf{Young modulus} & \textbf{Bending modulus} \\
	& $d = 2.13$ \AA & $Y = 21.5$ eV/\AA\textsuperscript{2} & $D = 1.46$ eV\\
	\hline

	\hline
	& \textbf{Eq. bond length} & \textbf{Bond springs} & \textbf{Angular springs} \\
	\textbf{Chain 1} & $l_1 = \ 0.5\ \sigma$ & $k_1 = 130600\ \epsilon/\sigma^2$ & $k_{\theta, 1} = 764\ \epsilon$\\
	\textbf{Chain 2} & $l_2 \approx 0.498\ \sigma$ & $k_2 \approx 130039\ \epsilon/\sigma^2$ & $k_{\theta, 2} \approx 761\ \epsilon$\\
	\hline

	\hline
	\end{tabular}
\end{table}

Note that the bond and angular spring constants $k$ and $k_\theta$ for each chain are related to the Young modulus $Y$ and bending modulus $D$ of graphene as
\[
 k = \frac{d}{l} Y, \qquad k_\theta = \frac{d}{l} D,
\]
where $d$ is the width of the dimer row and $l$ the equilibrium length of the bonds in the atomistic chain. In the initial configuration, each chain is in its isolated equilibrium state, and their relative distance is chosen as $h = 1.063\ \sigma$. The total energy of the system is a function of all atomic positions $ \mathbf{R}^1_i$, $i = 1, \dots, N_1$ (bottom chain) and $ \mathbf{R}^2_j$, $i = 1, \dots, N_2$ (top chain) and also on the periodic cell length $L$, and can be divided into three components:
\begin{equation}
	U_{\mathrm{tot}} = U_{\mathrm{intra}}^1 + U_{\mathrm{intra}}^2 + U_{\mathrm{inter}}.
\end{equation}
The intralayer energies $U_{\mathrm{intra}}^1$, $U_{\mathrm{intra}}^2$ take the form
\begin{equation}
	\begin{aligned}
		U_{\mathrm{intra}}^\nu
		 = & \ \frac{1}{2}\sum_{i = 1}^{N_l} k_\nu  \left ( \Vert \mathbf{R}^\nu_{i+1} - \mathbf{R}^\nu_i \Vert - l_\nu \right )^2  \\
		& + \frac{1}{2}\sum_{i = 1}^{N_\nu} k_{\theta, \nu}   \left (  \mathrm{sin}^{-1} \mathrm{det}\left [ \frac{\mathbf{R}^\nu_{i+1} - \mathbf{R}^\nu_i }{\Vert \mathbf{R}^\nu_{i+1} - \mathbf{R}^\nu_i \Vert },\ \frac{\mathbf{R}^\nu_{i-1} - \mathbf{R}^\nu_i }{\Vert \mathbf{R}^\nu_{i-1} - \mathbf{R}^\nu_i \Vert }\right ]  \right )^2, \ \nu = 1,\ 2,
	\end{aligned}
\end{equation}
where it is understood that periodic boundary conditions are applied: $\mathbf{R}_0^{1}$, $\mathbf{R}_{N_1+1}^{1}$, etc. refer to the appropriate atomic position in the adjacent periodic images. The interlayer energy $U_{\mathrm{inter}}$ is
\begin{equation}
	U_{\mathrm{inter}} =  \sum_{i = 1}^{N_1} \sum_j 4 \epsilon \left ( \left ( \frac{\sigma}{\Vert \mathbf{R}^2_j - \mathbf{R}^1_i \Vert } \right )^{12} - \left ( \frac{\sigma}{\Vert \mathbf{R}^2_j - \mathbf{R}^1_i \Vert } \right )^6 \right ),
\end{equation}
 where the inner sum runs over all the neighbors of $\mathbf{R}^1_i$ in the top layer that fall initially within the cutoff radius of $29\ \sigma$, whether in the simulation cell or the periodic images.

The numerical relaxation of atomic positions and total length of this coupled system leads to the rippled shape presented in \cref{fig:SimulatedRipples} and the partition of energy presented in \cref{tab:EnergyChangeAtomistic}. The total length of the system is reduced by $0.188\% $.

\begin{table}[ht]
\caption{Overall change in energy components by numerical relaxation.}\label{tab:EnergyChangeAtomistic}
 \centering
\begin{tabular}{c|c|c}
\hline

\hline
\textbf{Bond springs} & \textbf{Angle springs} & \textbf{Lennard-Jones}\\ \hline
$+0.0358 \epsilon $ & $+0.5197 \epsilon$ & $-1.474 \epsilon$\\
\hline

\hline
\end{tabular}
\end{table}

A first observation from this computation is that the energy reduction achieved by bending appears to be about twice the potential energy in the angle springs in the relaxed configuration. Moreover, the two chains relax nearly inextensibly by producing a coherent ripple in the vertical direction by bending together. The in-plane strain energy, modeled here by the potential energy in the bond springs, is an order of magnitude smaller than the other energies.

\begin{figure}[t]
	\centering
	\includegraphics[width=\textwidth]{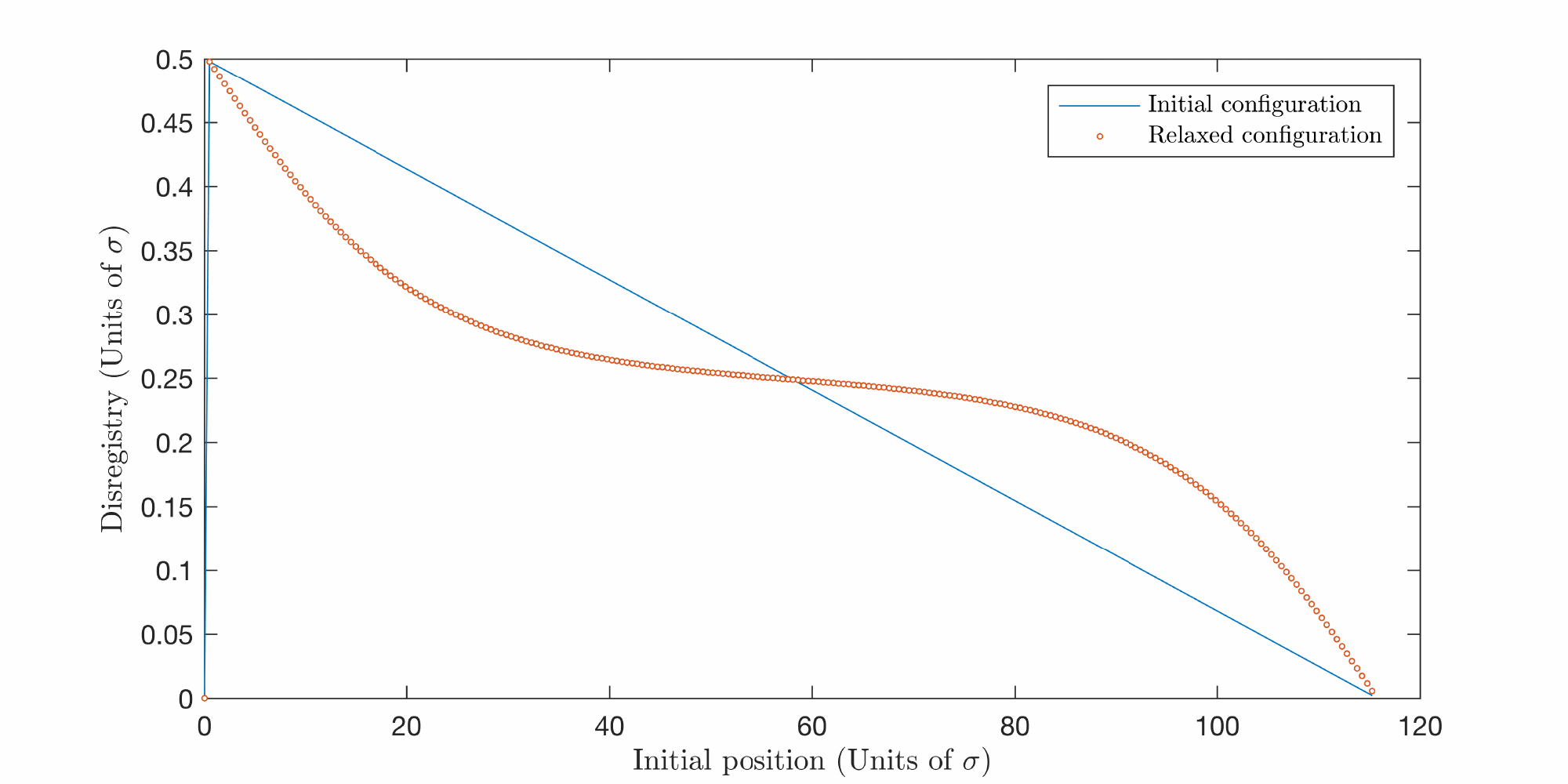}
	\caption{Computed disregistry from simulation of nearly commensurate chains.}
	\label{fig:SimulatedDisregistry}
\end{figure}

 The driving mechanism for this relaxation process can be infered by measuring the \textit{disregistry} of the atoms in layer $2$ compared to the atoms in layer $1$. This is a measure of the local atomic configuration, obtained by projecting the position of each atom in layer $2$ on the curve describing layer $1$. By measuring the distance of the projection to the nearest atom to the left on this curve, we can distinguish between favorable (staggered) configurations, $\Delta = .25\ \sigma$, and unfavorable (aligned) configurations, $\Delta = 0$.

The disregistry obtained in the previous numerical simulation is plotted on \cref{fig:SimulatedDisregistry}. Clearly, the relaxation tends to increase the number of atoms close to the most favorable configuration. On the other hand, the distance between the layers does not vary much.

Our conclusion is that this relaxation process can be correctly modeled by considering only the \textit{bending} degrees of freedom and neglecting any in-plane strain in each layer. The driving mechanism behind the formation of coherent ripples is the competition between two main contributions:
\begin{enumerate}
	\item the bending energy in the angular bonds due to chain curvature;
	\item the interlayer potential energy variations between local configurations arising from the disregistry across the layers.
\end{enumerate}

% section motivating_numerical_example (end)

\section{Derivation of the model}\label{sec:DerivationModel}

We will now develop a mathematical model for the determination of the ground state of a coupled system of parallel periodic atomic chains which are infinite and truly \textit{incommensurate}: unlike the previous numerical example, there exists no periodic supercell for the coupled system. This is achieved by choosing an irrational number to be the ratio $l_2/l_1$ of equilibrium length for the intralayer bonds.

\subsection{Geometry}
First, we describe the model geometry of the coupled system. Our main assumption is that the chains act inextensibly due to the strong in-plane modulus, as observed numerically in \Cref{sec:motivating_numerical_example}. Consecutive atoms in each chain will be separated by a fixed distance, respectively $1$ for the chain $\mathcal{C}^1$ and $\alpha$ for chain $\mathcal{C}^\alpha$ where $\alpha$ is an irrational real number. The chains are allowed to bend in the two-dimensional plane, while they remain separated by a fixed distance $h$. Consequently, $\mathcal{C}^1$ and $\mathcal{C}^\alpha$ form two infinite smooth parametric curves in the plane separated by a fixed distance, as seen on \cref{fig:geomcurved:a}.

\begin{figure}[t]
\centering
\begin{subfigure}[b]{.59\textwidth}\centering
\includegraphics[width=.92\textwidth]{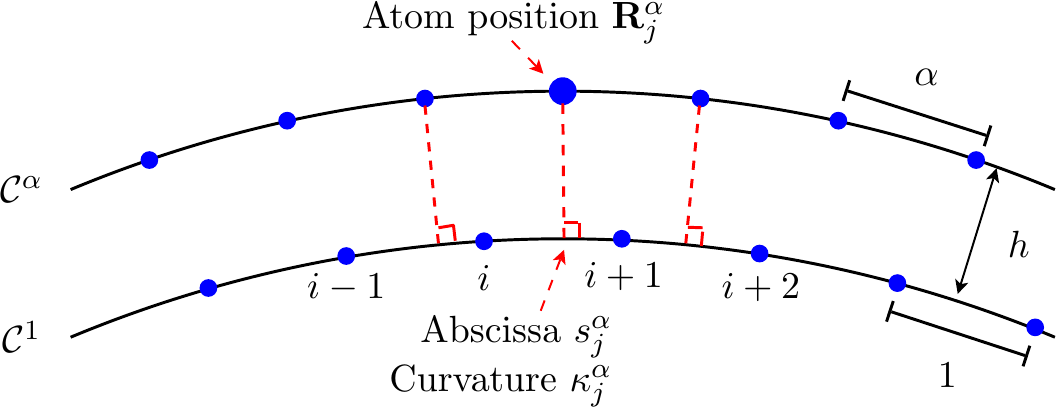}
\vspace{1cm}
\caption{Abscissa and curvature at projected positions.}\label{fig:geomcurved:a}
\end{subfigure}
\begin{subfigure}[b]{.4\textwidth}\centering
\includegraphics[width=.92\textwidth]{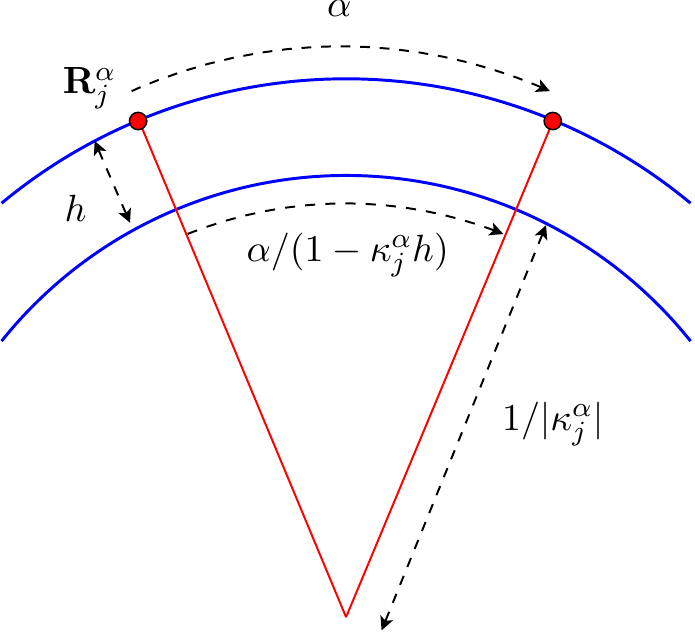}
\caption{Curvature-induced modulation.}\label{fig:geomcurved:b}
\end{subfigure}
\caption{Geometry of a curved system of coupled chains.}\label{fig:geomcurved}
\end{figure}

To derive the equations of the model, we assume that the curvature of the system remains small and varies slowly. Consequently, we use the arc length along either curve to measure the distance between atoms in the respective chain. Furthermore, we choose to describe the whole system by following the curve $\mathcal{C}^1$. This leads us to introduce the following quantities.

\paragraph{Abscissa}
Let $s \in \mathbb{R}$ denote the arc length parameter along the curve $\mathcal{C}^1$, and $\boldsymbol{\gamma}^1(s): \mathbb{R} \to \mathbb{R}^2$ be a natural parametrization of $\mathcal{C}^1$: $\Vert \boldsymbol{\gamma}_1'(s) \Vert = 1$. Atoms on the chain $\mathcal{C}^1$ are situated at $\mathbf{R}^1_i = \boldsymbol{\gamma}^1(i)$ for each $i \in \mathbb{Z}$. Let the atoms of chain $\mathcal{C}^\alpha$ be positioned at points $\{ \mathbf{R}^\alpha_j \}_{j \in \mathbb{Z}}$. Then, their projection on $\mathcal{C}^1$ can be parameterized by a strictly increasing sequence of abscissas as seen on \cref{fig:geomcurved:a}, which we denote
\begin{equation}
	\left \{s^\alpha_j \right \}_{j \in \mathbb{Z}} \qquad \text{s.t.}\quad \forall j \in \mathbb{Z}, \quad \mathbf{R}^\alpha_j = \boldsymbol{\gamma}^1(s^\alpha_j) + h \mathbf{n}^1(s^\alpha_j),
\end{equation}
where $\mathbf{n}^1(s)$ is the normal vector to the curve $\mathcal{C}^1$ at abscissa $s$.

\paragraph{Curvature}
While the arc length between consecutive atoms $\mathbf{R}^\alpha_j$ along $\mathcal{C}^\alpha$ is fixed as $\alpha$, this is not the case for their projections on $\mathcal{C}^1$.
As seen on \cref{fig:geomcurved:b}, the arc length between projections of consecutive atoms, i.e. $s^\alpha_{j+1} - s^\alpha_j$ for $j \in \mathbb{Z}$, depends on the curvature of the curve $\mathcal{C}^1$, which we denote $\kappa$. Note that we choose the convention $\kappa<0$ for the configurations presented in \cref{fig:geomcurved}. We obtain then, approximating locally the curvature by the constant value $\kappa^\alpha_j = \kappa( s^\alpha_j,\,s^\alpha_{j+1})$ :
\begin{equation*}
 s^\alpha_{j+1} =  s^\alpha_j + d^\alpha(s^\alpha_j) \qquad \text{ where } \quad d^\alpha(s^\alpha_j) = \frac{\alpha}{1 - h \kappa^\alpha_j}.
\end{equation*}
Assuming that $\vert h \kappa^\alpha_j \vert \ll 1$, we will make the first-order approximation
\begin{equation}
d^\alpha(s^\alpha_j) =  (1 + h \kappa^\alpha_j)\alpha.
\end{equation}
This leads to the formula for the curvature
\begin{equation}\label{eq:curvature}
h \kappa^\alpha_j =  \frac{s^\alpha_{j+1} - s^\alpha_j}{\alpha} - 1.
\end{equation}

\subsection{Potential energy}
To formulate our model, we consider only two contributions to the potential energy of the system of coupled chains: an interaction energy and a bending energy. For simplicity, the bending energy is modeled as a quadratic term in the curvature variable. It remains to construct a suitable approximation for the inter-chain potential energy as a function of the variables $\{ s_j^\alpha \}_{j \in \mathbb{Z}}$.

Let $V$ be an interatomic potential contributed by two atoms each in a different chain depending only on their relative distance, such as the Lennard-Jones potential. The total interaction energy is formally:
\begin{equation}\label{eq:periodicpotential1}
\mathcal{E}_{int} = \sum_{j} V_{1 \to \alpha}(\mathbf{R}^\alpha_j) \qquad \text{where} \quad V_{1 \to \alpha}(\mathbf{R}) = \sum_{i = -\infty}^\infty V( \Vert \mathbf{R}^1_i - \mathbf{R} \Vert ).
\end{equation}
In this expression, we have isolated the contribution of each atom in chain $\mathcal{C}^\alpha$ as an inter-chain potential energy created collectively by all atoms of chain $\mathcal{C}^1$.
Due to the long-range nature of $V$, this inter-chain potential depends not only on the abscissa $s^\alpha_j$ of the projection of the atom on the curve, but also on the shape of the whole curve as we see on \cref{fig:geomcurved:a}. However, due to the fast decay of the inter-atomic potential, the main contribution to the inter-chain potential is local and can be deduced by approximating positions of the neighboring atoms belonging to the curved chain $\mathcal{C}^1$.

\begin{figure}[t]
\centering
\includegraphics[width=.6\textwidth]{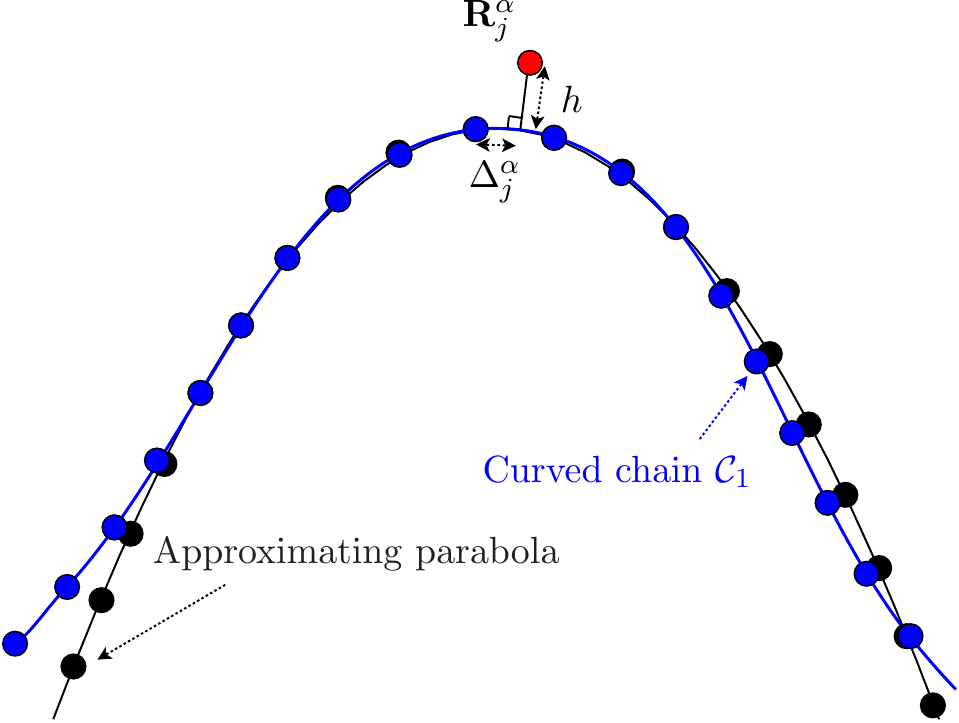}
\caption{Parabolic approximation for the curved chain $\mathcal{C}^1$ near atom $j$ of chain $\mathcal{C}^\alpha$.} \label{fig:ParabolicApproximation}
\end{figure}
Let us construct such an approximation explicitely. For each $j \in \mathbb{Z}$, a parabola is used as a locally second-order accurate approximation $\widehat{\mathcal{C}}^1$ to the exact curve $\mathcal{C}^1$, in the neighborhood of the atom positioned at $\mathbf{R}_j^\alpha$. This parabola, parametrized as $s \mapsto \widehat{\boldsymbol{\gamma}}^1(s)$, is uniquely determined (see \cref{fig:ParabolicApproximation}) by the distance $h$ between the curves and the local curvature $\kappa^\alpha_j$ of curve $\mathcal{C}^1$, choosing $\mathbf{R}_j^\alpha$ as origin of a local system of coordinates. The approximate atomic positions $\widehat{\mathbf{R}}^1_i = \widehat{\boldsymbol{\gamma}}^1(i)$ along this parabola are then deduced from the additional input of the projected abscissa $s_j^\alpha$, using the arc length $\Delta_j =  s_j^\alpha \ (\text{mod }1)$. % between the projection of $\mathbf{R}^\alpha_j$ on $\mathcal{C}^1$ and the nearest atom on the left belonging to $\mathcal{C}^1$.

This geometrical construction allows us to introduce now a disregistry-- and curvature--dependent approximate inter-chain potential. For given values of $s, \kappa \in \mathbb{R}$ and $h>0$, we obtain the sequence of positions $\{\widehat{\mathbf{R}}^1_i\}_{i \in \mathbb{Z}}$ as above, relative to the origin $\mathbf{R}$ of the coordinates. Then, we propose as an approximation to $V_{1 \to \alpha}(\mathbf{R})$ introduced in~\cref{eq:periodicpotential1} the quantity:
\begin{equation}\label{eq:periodicpotential2}
\widehat{V}_{\mathrm{per}}(s ; \kappa, h) = \sum_{i = -\infty}^\infty V(\Vert  \widehat{\mathbf{R}}^1_i - \mathbf{R} \Vert).
\end{equation}
It is clear that the potential $\widehat{V}_{\mathrm{per}}(s ; \kappa, h)$ is a periodic function of the variable $s$ with period $1$. Note that in the following, we will drop the dependency of $\widehat{V}_{\mathrm{per}}$ on the fixed variable $h$.
\begin{remark}
This construction of $\widehat{V}_{\mathrm{per}}$ gives an explicit formula for the approximate potential. Hence, it can be used to precompute numerically and accurately an approximate inter-chain potential, for calculations aimed at finding ground state configurations.
\end{remark}
\noindent The total potential energy of the full system is then modeled as the formal sum:
\begin{equation}\label{eq:potentialenergy1}
U \left ( \{ s^\alpha_j \} \right ) = \sum_j \left ( \widehat{V}_{\mathrm{per}}\left (s^\alpha_j  ; \kappa^\alpha_j \right ) + \frac{\beta}{2} \left ( \kappa^\alpha_j\right )^2 \right ),
\end{equation}
where $\beta$ is the angle spring constant modeling the resistance to bending of the coupled system of chains. Thanks to Eq.~\cref{eq:curvature}, this energy is formally
\begin{equation}\label{eq:potentialenergy2}
U \left ( \{ s^\alpha_j \} \right ) = \sum_j  v(s^\alpha_j, s^\alpha_{j+1}),
\end{equation}
where $v(s^\alpha_j, s^\alpha_{j+1})$, the local potential energy at site $j$, is given by
\begin{equation}\label{eq:siteenergy}
v(s, t) = \widehat{V}_{\mathrm{per}}\left (s\  ; \ \left ( \frac{t - s}{\alpha } - 1 \right ) h^{-1}\right ) + \frac{\beta}{2h^2} \left (  \frac{t- s}{\alpha} -1 \right )^2.
\end{equation}

\section{Ground state analysis}\label{sec:GroundStateAnalysis}
The model described by Eq.~\cref{eq:potentialenergy1} or~\cref{eq:potentialenergy2} is a generalized Frenkel-Kontorova model. There exists a vast literature on the Frenkel-Kontorova model, which we will not attempt to summarize here: see e.g.~\cite{BraunKivshar2004frenkel}. We only present the basic theory necessary for our analysis.

 Following~\cite{AubryLeDaeron1983}, we will make the following assumptions on the function $v$:
\begin{enumerate}
\item $v(s,t)$ has a lower bound: there exists $B$ such that
\begin{equation}\label{eq:LowerBoundCondition}
	B \leq v(s, t), \qquad \forall s, t \in \mathbb{R}.
\end{equation}
\item $v(s, t)$ is invariant by translation by $(1,1)$ and satisfies a symmetry condition:
\begin{equation}\label{eq:PeriodicityCondition}
	v(s+1, t+1) = v(s, t), \quad  v(s, t) =  v(-t, -s),  \qquad \forall s, t \in \mathbb{R}.
\end{equation}
\item $v(s, t)$ is a twice continuously differentiable function, and satisfies for all $s,t \in \mathbb{R}$,
\begin{equation}\label{eq:TwistCondition}
	-\frac{\partial^2 v}{\partial s \partial t}(s, t) > C > 0,
\end{equation}
where $C$ is a finite positive constant.
\end{enumerate}
Note that for the function $v$ defined by~\cref{eq:siteenergy}, conditions~\cref{eq:LowerBoundCondition} and~\cref{eq:PeriodicityCondition} are satisfied whenever $\widehat{V}_{\mathrm{per}}$ is bounded from below. For the twist condition~\cref{eq:TwistCondition} to hold, stronger conditions are necessary, namely $\widehat{V}_{\mathrm{per}}$ must be twice continuously differentiable with $\frac{\partial^2 v}{\partial s \partial t}$ bounded from below and $\beta$ large enough.

Under these conditions, Aubry-Mather theory leads to the following description and results, for which details are to be found in the seminal paper~\cite{AubryLeDaeron1983}:
\begin{definition}
A minimum energy configuration is a sequence $\{ s_j \}$ such that any finite change $\delta_j$ of a finite set of atoms necessarily increases the energy, that is:
\begin{equation}
\sum_{j=J'}^{J} \left ( v(s_j + \delta_j, s_{j+1} + \delta_{j+1}) -  v(s_j , s_{j+1})\right ) \geq 0
\end{equation}
for any $J' < J$, and any choice of $\delta_j$ with $\delta_j = 0$ for $j < J'$ and $j \geq J$.
\end{definition}
The set of minimum energy configurations, which includes all possible boundary conditions, can be further divided between {\it defect configurations} and {\it ground state configurations}. To make this distinction, one derives the equilibrium equations
\begin{equation}\label{eq:equilibrium}
\frac{\partial v}{\partial s_j} (s_j, s_{j+1}) + \frac{\partial v}{\partial s_j} (s_{j-1}, s_j) = 0.
\end{equation}
These equations can be interpreted as the motion equation of a dynamical system with discrete time $j$ by introducing the conjugate variable for $s_j$,
\begin{equation}
p_j = \frac{\partial v}{\partial s_j} (s_{j-1}, s_{j}).
\end{equation}
Then $( p_{j+1},s_{j+1})$ are recursively defined from $(p_j, s_j)$ by the implicit system derived from~\cref{eq:equilibrium}:
\begin{equation}\label{eq:dynamicalsystem}
\left \{ \begin{aligned}
\frac{\partial v}{\partial s_j}(s_j, s_{j+1}) &= - p_j, \\
p_{j+1} &= \frac{\partial v}{\partial s_{j+1}}(s_j, s_{j+1}).
\end{aligned} \right.
\end{equation}
Note that under the assumptions above, $\frac{\partial v}{\partial s}(s,t)$ is a monotonic function of $t$ for fixed $s$, and hence system~\cref{eq:dynamicalsystem} has a unique solution for a given $(s_j, p_j)$. One then introduces the bounded variable in the unit torus $\mathbb{T}$,
\begin{equation}
	\theta_j \equiv s_j \ (\text{mod }1),
\end{equation}

and then defines a non-linear operator $\widetilde{T}$ on the cylinder $ \mathbb{R} \times \mathbb{T}$ onto the cylinder,
\begin{equation}
\begin{pmatrix}  p_{j+1}  \\ \theta_{j+1} \end{pmatrix} = \widetilde{T} \begin{pmatrix}  p_{j} \\ \theta_j \end{pmatrix}.
\end{equation}
By construction, $\widetilde{T}$ is an area-preserving one-to-one map. We then define the orbit of a point $ (p_0, \theta_0)$ as an infinite sequence $(p_j, \theta_j) = \widetilde{T}^j (p_0, \theta_0)$ for $-\infty < j < \infty$.
\begin{definition}
	An orbit $(p_j, \theta_j)$ is recurrent if there exist an integer sequence $j_k \to \infty$ such that
	\[
		\lim_{k \to \infty }\left ( p_{j_k}, \theta_{j_k} \right ) = \left (p_0, \theta_0\right ).
	\]
\end{definition}
% \begin{remark}
% By definition, {\it stationary configurations} will be equivalently identified as orbits of the dynamical system defined by $\widetilde{T}$, and thus represented uniquely by the initial point $(p_0, \theta_0)$ modulo an integer translation. Indeed, atomic positions are recursively determined from the sole knowledge of the initial point $(p_0, \theta_0)$ as $(p_j, \theta_j) = \widetilde{T}^j (p_0, \theta_0)$. Note that the set of points that represent minimum energy configurations is invariant by $\widetilde{T}$.
% \end{remark}
\begin{definition}\noindent
\begin{enumerate}
\item A defect configuration $\{ s_j \}$ is a stationary configuration (satisfying~\cref{eq:equilibrium}) such that the associated orbit $\{ \widetilde{T}^j (p_0, \theta_0)\}$ is non-recurrent.
\item An elementary defect is a defect configuration which is also a minimum energy configuration.
\item A ground-state configuration is a minimum energy configuration which is not a defect configuration (the associated orbit is recurrent).
\end{enumerate}
%\hl{ML: We discuss ground-state configurations below, but should we remark about what defects %look like?}
\end{definition}
Using the previous definitions, the following two theorems, which were first stated and proved by Aubry and Le Daeron in~\cite{AubryLeDaeron1983}, largely identify the minimum energy configurations and ground states of the Frenkel-Kontorova model:
\begin{theorem}\noindent
\begin{enumerate}
\item
For any minimum energy configuration, there exists two real numbers $l$ (the atomic mean distance) and $\omega$ (the phase), such that for any $j$, $s_j$ and $j l + \omega$ belong to the same interval $[k_j/2 , (k_j+1)/2]$ where $k_j = \mathrm{Int} (2 s_j)$.
\item For any value of $l$, there exists minimum energy configurations $\{ s_j \}$ such that
\[
\lim_{\vert j - k \vert \to \infty} \frac{s_j - s_k}{j - k} = l.
\]
\end{enumerate}
\end{theorem}
Note that the number $l$ is also called the rotation number of the orbit corresponding to a minimum energy configuration $\{ s_j \}$. While it can in principle be arbitrarily chosen, in the case of the present model we are interested mainly in the choice $l=\alpha$, which models the case of a bilayer system which is flat on average. Other choices could be interpreted as systems that have a global curvature.

\begin{theorem} \label{th:AubryLeDaeron}\noindent
\begin{enumerate}
\item The set of minimum energy configurations with atomic mean distance $l$ is totally ordered, i.e., for two such configurations $\{ s_j \}$ and $\{ t_j \}$, one of the assertions holds true for all $j$: $s_j < t_j$, $s_j = t_j$, or $s_j > t_j$.
\item The whole set of ground state configurations with irrational atomic mean distance $l$ can be parameterized with one or two {\it hull functions} $f$ which are strictly increasing:
\begin{enumerate}
\item When $f$ is continuous, a unique function $f$ parameterizes all ground states;
\item When $f$ is discontinuous, two determinations $f^+$ and $f^-$ are necessary to parameterize all ground states. They are both strictly increasing and correspond to the right continuous or left continuous determinations of the same discontinuous function. (In this case, the set of discontinuity points of $f^\pm$ is dense on $\mathbb{R}$.)
\item Functions $x \mapsto g^\pm(s) = f^\pm(s) - s$ are $1$-periodic in $\mathbb{R}$.
\item Any ground-state configuration is determined by a phase $\omega$ and a determination of $f$, either $f^+$ or $f^-$ such that
\begin{equation}\label{eq:hullfunction}
s_j = f(j l + \omega).
\end{equation}
\end{enumerate}
\end{enumerate}
\end{theorem}

In particular, incommensurate ground-states are represented by trajectories which are rotating either on a smooth Kolmogorov-Arnold-Moser torus (when $f$ is continuous) or on a discontinuous Cantor set or Cantorus (when $f$ is discontinuous). A striking consequence is that ground-state trajectories always belong to the possibly zero-measure set of non-chaotic trajectories, in particular the ground-state ensemble is described by a Cantorus typically embedded in a chaotic region of the map. The transition between these two regimes has been called breaking of analyticity~\cite{AubryLeDaeron1983,aubry1983devil,de1984critical}.

\section{Approximation by periodic configurations: a partial justification}\label{sec:ApproximationPeriodic}

Incommensurate ground-states are typically constructed as the limit of commensurate ground-states~\cite{aubry1983devil,AubryLeDaeron1983,de1984critical} with rotation number $\alpha_n = p_n/q_n$ converging to the irrational $\alpha$. However, in general only a subsequence of such configuration sequences converge.

In fact, we will see that when there exists a smooth invariant torus for $\widetilde{T}$ with rotation number $\alpha$, then it is the limit (e.g. in the Hausdorff metric) of such periodic orbits; however, when the set of ground states with rotation number $\alpha$ forms a Cantor set, one may only prove that it is the limit of a particular sequence of commensurate ground states. For example, a given sequence may in principle converge to an incommensurate defect configuration, i.e., an orbit which is asymptotic to the Cantorus in both directions.

A validation of this approach, at least in the case of a smooth hull function, can be obtained by the use of perturbation theory. The starting point is given by the following results, see e.g.~\cite{Falcolini1992, Tompaidis1996} for the proof:
\begin{proposition}\label{prop:normal_form}
Assume that $\widetilde{T}: \mathbb{T}\times \mathbb{R}$ is a 2-dimensional symplectic map of class $C^r$ with $r > 1$. Suppose that $\widetilde{T}$ admits a $C^r$ (resp. analytical) invariant circle $\Gamma$, homotopic to $\mathbb{T}\times \{0\}$, on which the motion is $C^r$-conjugate  (resp. analytical) to a rigid rotation with rotation number $\alpha$ Diophantine of type $(K,\tau)$, i.e. such that
\[
\vert q \cdot \alpha \vert > \frac{K}{q^\tau},\qquad \forall q \in \mathbb{N}^*.
\]
Then, given any nonnegative integer $k < (r-1)/\tau$, there exists a symplectic $C^{r-1-k \tau}$  (resp. analytical) mapping $\Psi$ defined in a neighborhood of $\Gamma$, mapping $\Gamma$ to $\mathbb{T}\times \{0\}$, and having a $C^{r-1-k \tau}$  (resp. analytical) inverse in a neighborhood of $\Gamma$, such that
\begin{equation}\label{eq:normal_form}
\psi \circ \widetilde{T} \circ \psi^{-1} (\Phi, \bar{H}) = (\Phi + \alpha +  \bar{H}\Delta(\bar{H}), \ \bar{H}) + R(\Phi, \bar{H}),
\end{equation}
where $\Delta: \mathbb{R} \to \mathbb{R}$ and $R: \mathbb{T} \times \mathbb{R} \to \mathbb{T} \times \mathbb{R}$ are $C^{r-1-k \tau}$  (resp. analytical) functions, $\Delta(0) \neq 0$, and the remainder $R$ satisfies in a neighborhood of $\Gamma$,
\[
\Vert R \Vert \leq C_k \vert \bar{H} \vert^{k+1}.
\]
\end{proposition}
\begin{remark}
The normal form~\cref{eq:normal_form} is obtained in practice by constructing an integral $H$ which is approximately invariant by $\widetilde{T}$, up to $\mathcal{O}(\vert \bar{H} \vert^{k+1})$.
\end{remark}
Using the mapping $\psi$, we have thus a new set of coordinates $(\Phi,\bar{H})$ defined in a neighborhood of $\Gamma$, and such that $\Gamma$ is the curve $\bar{H} = 0$. In particular, the fact that $\mathcal{T}$ is approximately integrable in the vicinity of $\Gamma$ leads to the following result:
\begin{lemma}\label{lemma:strip}
If $\vert \alpha - p/q \vert $ is small enough, all periodic orbits of type $p/q$ are contained in the strip
\[
\vert \bar{H} \vert \leq K \left \vert \alpha - \frac{p}{q} \right \vert,
\]
where $K$ depends only on the system and on the circle.
\end{lemma}
\begin{proof}
The idea of the proof, see~\cite{Falcolini1992} for details, is that since $\widetilde{T}$ is a small perturbation of an integrable mapping, there exists invariant K.A.M. tori whose rotation number $\alpha'$ is such that $\vert \alpha' - p/q \vert \leq \vert \alpha - p/q \vert $ and $p/q$ belongs to the interval $(\alpha, \alpha')$. Then, because the map which to a rotation number associates the corresponding invariant circle is Lipschitz, the circle of rotation number $\alpha'$ is contained in the strip
\[
\vert \bar{H} \vert \leq K/2 \vert \alpha - \alpha' \vert \leq K \left \vert \alpha - \frac{p}{q} \right \vert.
\]
Finally, as a consequence of the twist property, the periodic orbit of type $p/q$ has to be contained between the circles of rotation number $\alpha'$ and $\alpha$.
\end{proof}
We can now easily show that the computation of periodic ground states allows us to approximate the incommensurate ground states, i.e., the associated hull function. To do so, let us consider a periodic minimum energy configuration, characterized equivalently by a finite sequence $s_0, \dots, s_{q-1}$ and $s_q = s_0 + p$, or the initial point $(p_0, \theta_0)$. Without loss of generality, we choose $s_0 \in [0,1)$.

We define a function $F_q: [0,1] \to [s_0,s_0+1]$ by associating:
\begin{itemize}
\item to each point $x_j = \{ j \frac{p}{q}\} $ the value $\{ s_j - s_0\} + s_0$ for $j =0, \dots, q-1$, where $\{ \cdot \}$ denotes the fractional part,
\item to $x_q = 1$ the value $s_0 + 1$,
\item finally, the periodic function $x \mapsto F_q(x) - x$ is written as the linear interpolant of the previously defined values:
\begin{equation}\label{eq:PeriodicHullFunction}
	F_q: x \mapsto s_{j} +  \frac{x - x_j}{x_{j+1} - x_j} (s_{j+1} - s_j)\qquad \text{ for } x_j \leq x \leq x_{j+1}.
\end{equation}
\end{itemize}
This function can then easily be extended to $\mathbb{R}$ by the rule $F_q(x+1) = F_q(x)+1$.
\begin{proposition}\label{prop:HullFunctionApproximation}
Let the mapping $\widetilde{T}$ and a smooth invariant circle $\Gamma$ be as in \cref{prop:normal_form} with $k \geq 1$. Then, for $\vert \alpha - p/q \vert$ small enough, there exists a constant $C$ depending only on the system and on the circle, such that if $f$ is a hull function for the set of ground states associated with $\Gamma$, there exists a phase $\omega \in \mathbb{R}$ such that
\begin{equation}\label{eq:HullFunctionApproximation}
\left \Vert F_q - f(\cdot + \omega) \right \Vert_\infty \leq C \left ( \left \vert \alpha- \frac{p}{q} \right \vert + \frac{1}{q^2} \right ).
\end{equation}
\end{proposition}
\begin{remark}
When $\alpha$ is of type $(K,\tau)$ with $\tau$ large, there are values of $q$ for which the dominant error term is the interpolation error, $O \left ( 1/q^2 \right )$. The error bound can then be improved by using e.g. spline or trigonometric interpolation.
\end{remark}
\begin{proof}
Let $(\Phi_0, \bar{H}_0), \dots, (\Phi_{q-1}, \bar{H}_{q-1})$ be the points of the orbit of type $p/q$ in the new set of coordinates around $\Gamma$, assuming that $\vert \alpha - p/q \vert$ is small enough.

From the normal form of the mapping~\cref{eq:normal_form}, we deduce that for $j = 0, \dots, q-1$,
\[
\left \Vert \left ( \Phi_j + q \left (\omega + \bar{H}_j \Delta(\bar{H}_j) \right ), \ \bar{H}_j \right ) - (\Phi_j, \bar{H}_j,) \right \Vert \leq q C_k \max_i \vert \bar{H}_i \vert^{k+1}.
\]
By \cref{lemma:strip}, we deduce then that
\[
\left \vert \omega + \bar{H}_j \Delta(\bar{H}_j)  - \frac{p}{q} \right \vert \leq K C_k \left \vert \alpha - \frac{p}{q} \right \vert^{k+1},
\]
and by the triangular inequality,
\[
\left \vert\bar{H}_j \Delta(\bar{H}_j)  - \bar{H}_0 \Delta(\bar{H}_0)  \right \vert \leq 2 K C_k \left \vert \alpha - \frac{p}{q} \right \vert^{k+1}.
\]
Since $\Delta$ is smooth and nonzero in a neighborhood of $\Gamma$, it follows that for some $C>0$,
\[
\left \vert\bar{H}_j  - \bar{H}_0 \right \vert \leq C \left \vert \alpha - \frac{p}{q} \right \vert^{k+1}.
\]
Next, we have
\[
\left \Vert \left ( \Phi_0 + j \left (\omega + \bar{H}_0 \Delta(\bar{H}_0) \right ), \ \bar{H}_0 \right ) - (\Phi_j, \bar{H}_j,) \right \Vert \leq C \left \vert \alpha - \frac{p}{q} \right \vert^{k+1},
\]
and using the previous estimates we obtain for some $C>0$, since $k \geq 1$,
\begin{equation}\label{eq:PeriodicOrbitProjection}
\left \vert  \psi^{-1}\left ( \Phi_0 + j\frac{p}{q},\ 0 \right ) - \psi^{-1}\left (\Phi_j , \bar{H}_j \right )\right \vert \leq C  \left \vert \alpha- \frac{p}{q} \right \vert.
\end{equation}
Let $\pi_1$ be the first projection on the angle variable $s$ in the cylinder, $f$ be a hull function for $\Gamma$, $\omega \in \mathbb{R}$ such that for all $x \in \mathbb{R}$,
\[
 f(\omega + x) \equiv \pi_1 \left ( \psi^{-1}\left ( \Phi_0 + x, 0 \right ) \right ) \quad (\text{mod } 1), \qquad \text{and} \qquad  f(\omega) = s_0  \in [0,1).
\]
By construction, it is enough to prove~\cref{eq:HullFunctionApproximation} for $x \in [\omega, \omega + 1]$. Now since $p$ and $q$ are coprime, for $i \in \{ 0 \dots q-1\}$ , there exists $j \in \{ 0 \dots q-1\}$ such that
\[
\frac{i}{q} \equiv j \frac{p}{q} \quad (\text{mod } 1), \quad \text{and} \quad \left \vert F_q \left (\frac{i}{q} \right ) - f \left (\omega + \frac{i}{q} \right ) \right \vert \leq C \left \vert \alpha- \frac{p}{q} \right \vert,
\]
where the last estimate follows from the construction of $F_q$ and~\cref{eq:PeriodicOrbitProjection}. Since $f$ belongs to $C^r$ and by linearity of the interpolation, we obtain that the rate of convergence of the linear interpolant satisfies
\[
\left \Vert F_q(\cdot) - f(\cdot + \omega) \right \Vert_{\infty} \leq C \left ( \left \vert \alpha- \frac{p}{q} \right \vert + \frac{1}{q^2} \right ).
\]
\end{proof}

This result provides us with an \textit{a priori} estimate for the convergence rate of the supercell approximations. It justifies in particular the use of the so-called artificial strain as a measure of the error introduced the system by stretching one or both of the chains to obtain a commensurate system. 

Let us stress that the estimate is valid only in the case where the mapping $\widetilde{T}$ has a smooth invariant circle $\Gamma$ with rotation number $\alpha$, or equivalently, when the hull function is smooth (see \cref{th:AubryLeDaeron}). In general, when the bending modulus $\beta$ decreases enough, the system will undergo a commensurate-incommensurate transition at a value $\beta_{c}$ which depends on $\alpha$. The hull function then becomes discontinuous and the estimate breaks down.  

\section{Numerical examples}\label{sec:NumericalExamples}
We propose now to illustrate the previous analysis by some numerical examples, both in the smooth and the discontinuous case.
For the purpose of computations in this section, we use the following set of parameters:
\begin{table}[h!] \centering
\caption{Set of parameters used for the numerical computations}\label{tab:MyNumericalParameters}
\begin{tabular}{c|c|c|c|c|c|c} \hline

\hline
$\alpha$ & $p_{m}$ & $q_{m}$ &  $\epsilon$ & $\sigma$ & $h$ & $\beta$\\ \hline
$ \dfrac{8}{13}\dfrac{1+\sqrt{5}}{2} \approx 0.9957$ \vspace{.1cm} & $2555$ & $2566$ & $1$ & $2$ & $2.1262$ & $\{ 764, \ 10 \}$
\\ \hline

\hline
\end{tabular}
\end{table}

\noindent where we recall that the distance between atoms is $1$ for chain $\mathcal{C}^1$ and $\alpha$ for chain $\mathcal{C}^\alpha$, the distance between chains is $h$ and the inter-atomic potential between atoms of different chains is chosen as the Lennard-Jones potential
\[
V(r) = 4 \epsilon \left ( \left ( \frac{\sigma}{r} \right )^{12} - \left ( \frac{\sigma}{r} \right )^6 \right ).
\]
The ratio $\alpha$ is chosen here arbitrarily to satisfy the following conditions: first, the two chains have almost the same period, $\alpha \approx 1$, so we expect a large Moir\'e pattern and the formation of ripples, and second, $\alpha$ is of Diophantine type, see \cref{prop:normal_form}. Both conditions are satisfied here since the golden ratio $(1+\sqrt{5})/2$ is of Diophantine type $(\sqrt{5}, 2)$, see e.g.~\cite{hardy1979introduction}, and $13/8$ is one of its rational approximant.

\subsection{Realistic bending modulus: smooth ripples} % (fold)
\label{sub:realistic_bending_modulus_smooth_ripples}

We present first the numerical results when choosing the realistic value $\beta = 764$ for the bending modulus $\beta$ (see \Cref{sec:motivating_numerical_example} and~\cite{kudin2001c,IliaPaper}). The conditions~\cref{eq:LowerBoundCondition,eq:TwistCondition} are satisfied in this case. The numerical procedure is as follows. For any choice of rational approximant $\alpha \approx p/q$, we set up a periodic system
\begin{equation}
\left \{ s^{p,q}_j  \right \}_{j = 1 \dots q} \quad \text{ with initial values } \quad s^{p,q}_{j,0} = j \frac{p}{q}, \quad j = 1 \dots q,	
\end{equation}
and we minimize numerically the energy per periodic cell given by~\cref{eq:potentialenergy1}. From the relaxed values obtained for $\{ s^{p,q}_j  \}$, we can then obtain approximations $F_q$ for the hull function by Formula~\cref{eq:PeriodicHullFunction}. We can also reconstruct approximately the parametric curves modeling the chains using the curvature formula~\cref{eq:curvature} and plot the ripple geometry.

We compute first a reference solution by using the periodic approximant $p_m = 2555$, $q_m = 2566$ for which the error (also called artificial strain) is approximately 
\[
\vert \alpha - p_m/q_m \vert \leq  6 \cdot 10^{-8}.
\]
 Results are presented in \cref{fig:Hull_Function_Smooth:a,fig:Hull_Function_Smooth:b}. Note that due to the scaling difference of the two axis, the oscillations of the chains are greatly exaggerated in \cref{fig:Hull_Function_Smooth:a}, giving the false impression of a greatly varying distance between the chains (see the orthonormal view on top).
\begin{figure}[ht!]
\begin{subfigure}{\textwidth}\centering
	\includegraphics[width=.98\textwidth]{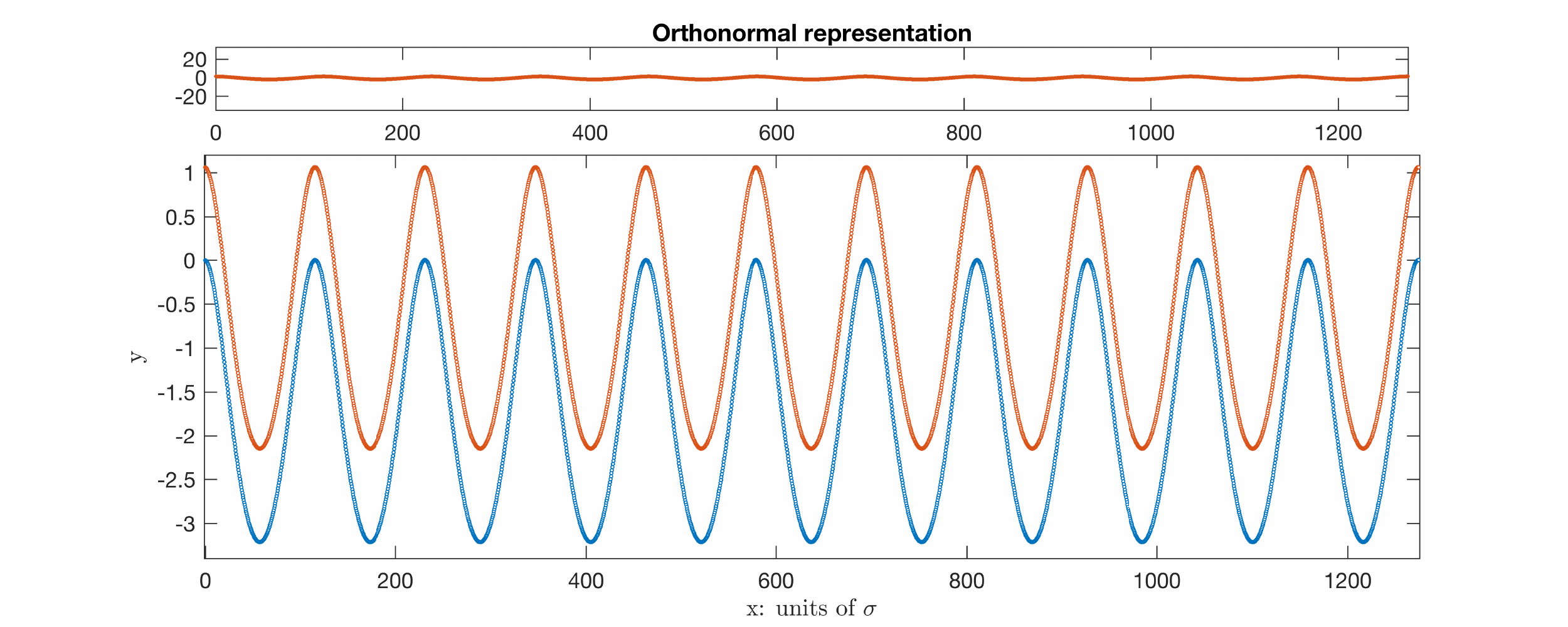}
	\caption{Reconstructed relaxed atomic chains}\label{fig:Hull_Function_Smooth:a}
\end{subfigure}
\begin{center}
	\begin{subfigure}{.49\textwidth}\centering
		\includegraphics[width=\textwidth]{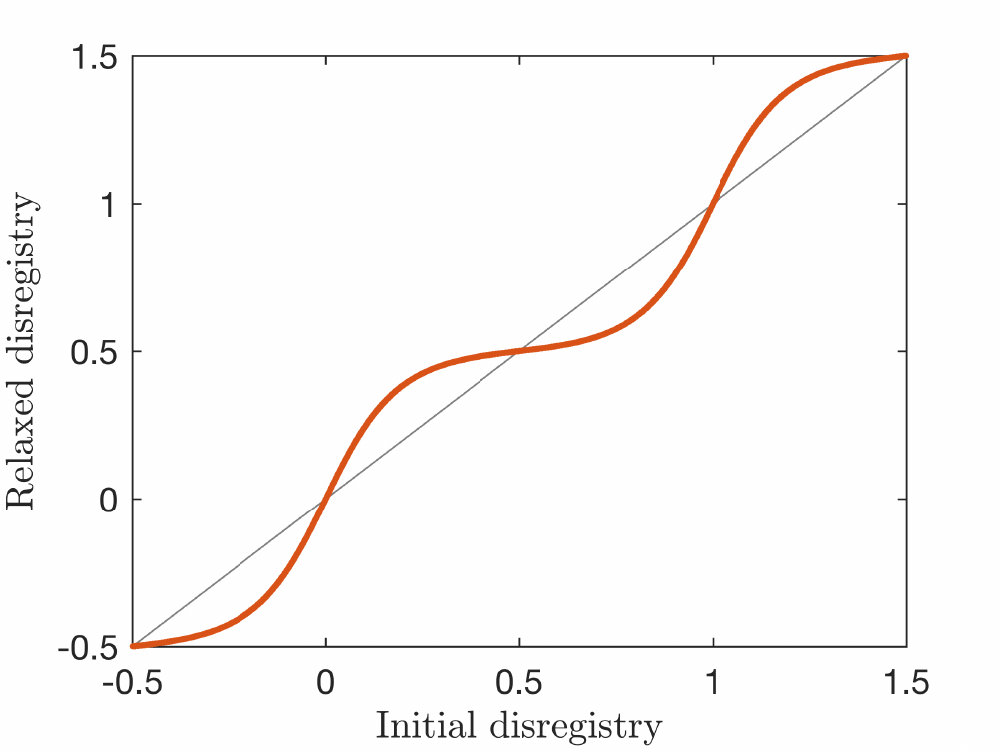}
		\caption{Hull function $F_q$}\label{fig:Hull_Function_Smooth:b}
	\end{subfigure}
	\begin{subfigure}{.49\textwidth}\centering
		\includegraphics[width=\textwidth]{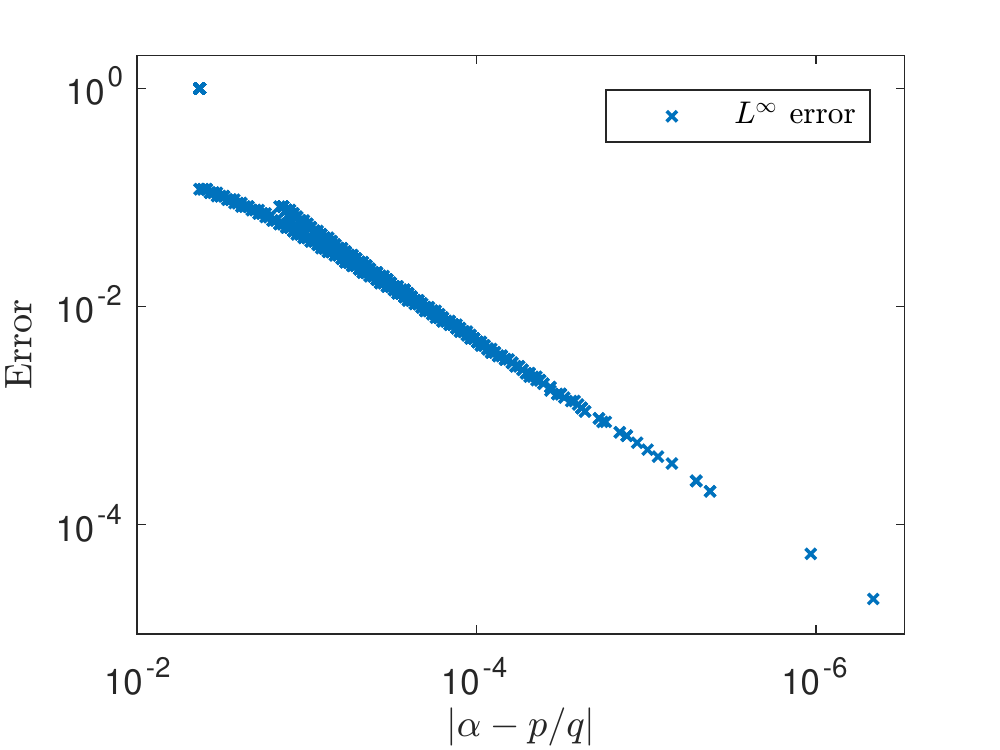}
		\caption{Estimated error for $F_q$}\label{fig:Hull_Function_Smooth:c}
	\end{subfigure}
\end{center}
\caption{Realistic case simulation results: $\beta = 764$}\label{fig:Hull_Function_Smooth}
\end{figure}
As expected, the relaxation process increases the number of atoms in a favorable staggered configuration, $s \approx .5\ (\textrm{mod}\ 1)$. We observe that both the shape of the ripples and the hull function show a smooth behavior. The overall shape of the twin curves seem almost perfectly periodic (even if the atoms are not distributed periodically on the curve) with the Moir\'e pattern period $ L_{M} \approx 106 \sigma$.  This indicates that a purely continuous approach is likely to be succesfuly in capturing the overall behavior, since atomic details (here, the atomic disregistry) integrate out at the mesoscopic scale.

Next, we evaluate numerically the convergence behavior of the numerical approach. We consider all periodic approximants $p/q$ to $\alpha$ with $q \leq 1000$, where $p$ is the closest integer to $\alpha q$, and we compute for each the relaxed configuration $ \{ s^{p,q}_j \}_{j = 1 \dots q}$. We evaluate then the error with respect to the reference configuration computed above, using the approximate error norm
\begin{equation}
	\Vert F_q - F_{q_m} \Vert_{q, \infty} = \min_{\omega \in [0,1)]}\max_{j = 1\dots q} \left \vert s^{p,q}_j - F_{q_m}\left ( j \frac{p}{q} + \omega \right ) \right \vert.
\end{equation}

The result of this calculation is presented in \cref{fig:Hull_Function_Smooth:c}. The computed error falls rapidly on a single straight line as a function of decreasing $\alpha - p/q$, showing good agreement with the estimate~\cref{eq:HullFunctionApproximation}. This supports the intuitive idea that the artificial strain $\vert \alpha - p/q\vert $ controls the error, rather than the size $q$ of the periodic supercell.

\begin{remark}
Note that the distance between layers $h$ above is determined by minimizing the energy with respect to $h$ only, e.g. for two flat layers, where the energy can be computed numerically.
\end{remark}
\begin{remark}
It is also possible to deal with a varying distance $h$, if we make the assumption that this distance is locally determined as a function of the position of the atom $S^\alpha_j$: indeed, this amounts to using a new potential given by the formula (see also~\cref{eq:periodicpotential2} and \cref{fig:ParabolicApproximation}):
\[
\widehat{V}_{\mathrm{per}}(s^\alpha_j; \kappa^\alpha_j) = \min_{h>0} \widehat{V}_{\mathrm{per}}(s^\alpha_j; \kappa^\alpha_j, h).
\]
\end{remark}

% subsection realistic_bending_modulus_smooth_ripples (end)

\subsection{Weak bending modulus: nonsmooth behavior} % (fold)
\label{sub:nonsmooth_behavior_for_a_weak_bending_modulus_value}
\begin{figure}[t!]
\begin{subfigure}{\textwidth}\centering
	\includegraphics[width=.98\textwidth]{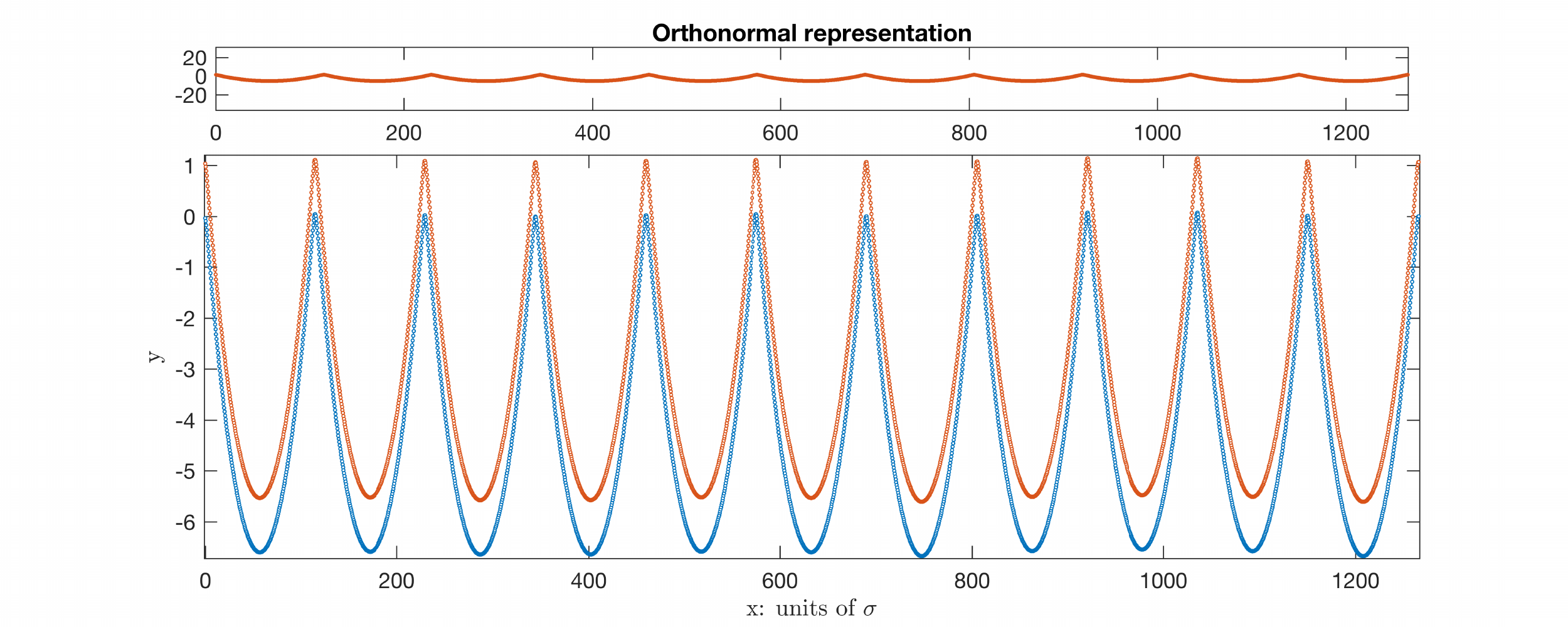}
	\caption{Reconstructed relaxed atomic chains}\label{fig:Hull_Function_NonSmooth:a}
\end{subfigure}
\begin{center}
	\begin{subfigure}{.49\textwidth}\centering
		\includegraphics[width=\textwidth]{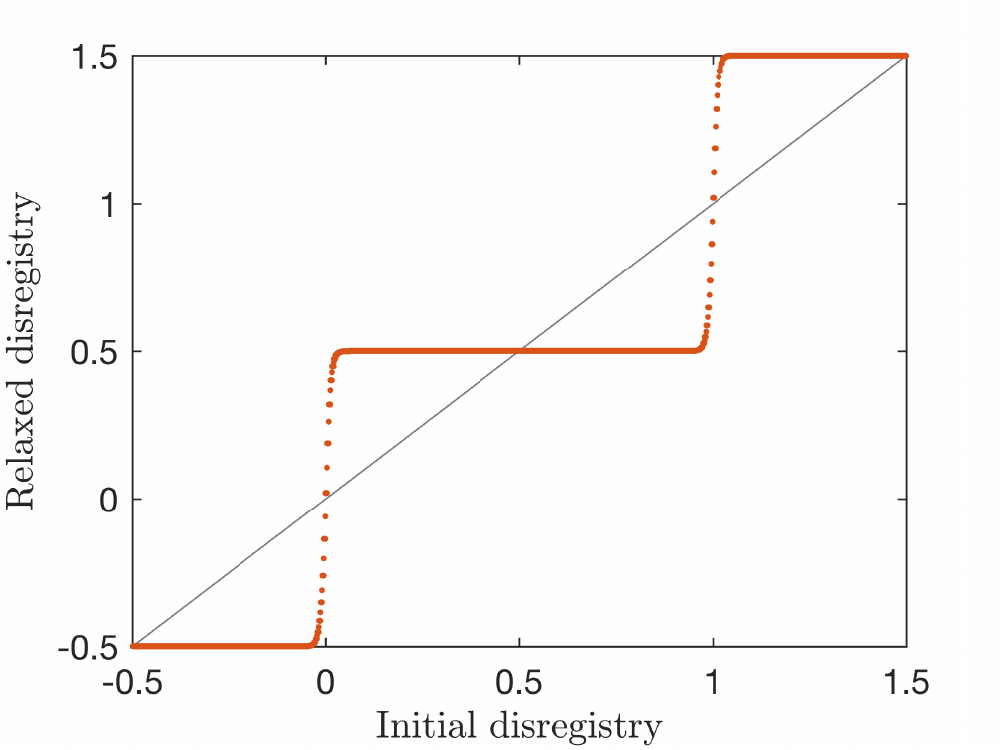}
		\caption{Hull function $F_q$}\label{fig:Hull_Function_NonSmooth:b}
	\end{subfigure}
	\begin{subfigure}{.49\textwidth}\centering
		\includegraphics[width=\textwidth]{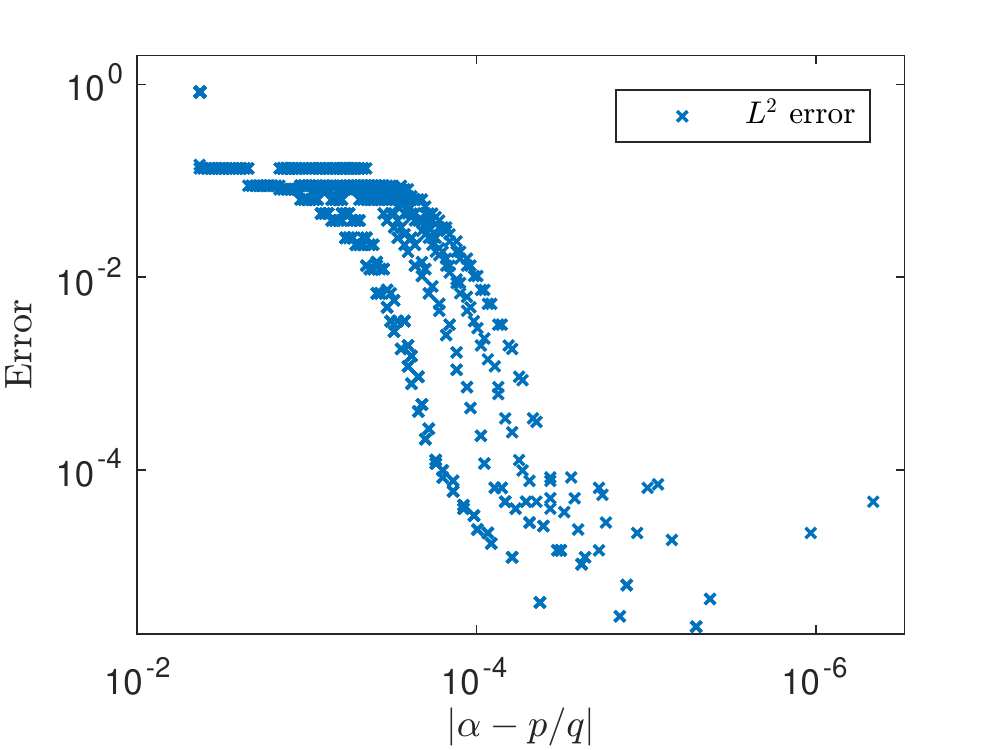}
		\caption{Estimated error for $F_q$}\label{fig:Hull_Function_NonSmooth:c}
	\end{subfigure}
\end{center}
\caption{Nonsmooth case simulation results: $\beta = 10$}\label{fig:Hull_Function_NonSmooth}
\end{figure}

Finally, we present the results for a weakened system of chains with bending modulus $\beta = 10$. Note that the twist condition~\cref{eq:TwistCondition} is still valid in this case. The new results are presented in \cref{fig:Hull_Function_NonSmooth}. The relaxation pattern in this case is quite different from the previous case: the ripples present a sharp peak separating regions of almost constant curvature, following the Moir\'e pattern. This allows the two chains to form commensurate regions where the atoms are all in the favorable staggered position, separated by \textit{solitons} or \textit{domain walls} which are classical in the study of the Frenkel-Kontorova model~\cite{AubryLeDaeron1983, BraunKivshar2004frenkel}.

The hull function is discontinuous in this case, with a large plateau at half-integer values corresponding to favorable staggered configurations and well-defined transition values forming smaller plateaus separated by forbidden gaps. This indicates a different mechanical behavior from the smooth case, as the sharp peaks are exponentially localized solitons, mechanically pinned by the Peierls-Nabarro potential~\cite{peierls1940size,nabarro1947dislocations,BraunKivshar2004frenkel}.

Another difference is a small, but nonzero secondary oscillation of the system of ripples, which can be observed by comparing the vertical position of the maxima of the reconstructed chains on \cref{fig:Hull_Function_NonSmooth:a}. This indicates the onset of the discrete interaction between solitons and substrate (Peierls-Nabarro potential) as well as between solitons at the mesoscopic scale. This is in contrast to the smooth case where atomic details could be integrated out in a continuum description at the mesoscale.

Numerically, there is also a marked difference in the convergence behavior. We observe on \cref{fig:Hull_Function_NonSmooth:c} that the computed errors do not fall onto a single line as in the smooth case. On average, we observe that the errors fall exponentially fast for $\vert \alpha - p/q \vert > 10^{-5}$, and then stagnate. Note that in this case, there is no smooth invariant circle with rotation number $\alpha$ for the map $\widetilde{T}$ (see \cref{sec:ApproximationPeriodic}), so the estimate~\cref{eq:HullFunctionApproximation} does not hold. Hence, in principle some subsequences of the periodic approximations may converge to metastable defect configurations.

% subsubsection nonsmooth_behavior_for_a_weak_bending_modulus_value (end)

\section*{Conclusion} % (fold)
In summary, we have presented a new model for the spontaneous relaxation for free-standing systems of incommensurate van der Waals bilayers. This relaxation phenomenon is driven by the trade-off between bending energy and the weak interlayer interactions and the associated disregistry-dependent local configuration potential. Note that this is a different mechanism from the widely studied intrinsic monolayer rippling observed, e.g., in suspended graphene~\cite{meyer2007structure}, which can be explained by phononic thermal fluctuations~\cite{fasolino2007intrinsic} and stress~\cite{bao2009controlled}. It differs also from ripples caused by interactions with a periodic substrate~\cite{ni2012quasi}.

To study mathematically the relaxation of such systems, we have introduced a new one-dimensional double chain model in which consecutive atoms in each chain are constrained to remain at a fixed distance. Relaxation occurs through spontaneous bending of the chains to minimize the interlayer interaction energy. This model can be seen as a new application of the generalized Frenkel-Kontorova model~\cite{BraunKivshar2004frenkel}.

As a result, we have identified the ground state configurations of the system, and we have given rigorous error estimates for the application to our model of the popular supercell approach used to approximate incommensurate systems. We have also presented numerical results supporting the analysis, both in the case of a realistic bending modulus where smooth quasi-periodic ripples can be numerically observed, and for a weaker bending modulus where commensurated domains separated by sharply peaked solitons appear. The transition between the two regimes is the famous commensurate-incommensurate transition~\cite{AubryLeDaeron1983}.

The next step will be to extend this model to the two-dimensional case. A particular difficulty is that perfectly inextensible two-dimensional sheets cannot be curved in more than one direction, and thus the model should be extended to account for small, but nonzero, in-plain strains.

% section conclusion (end)

\section{Acknowledgments}
This work was supported in part by ARO MURI Award W911NF-14-1-0247.  Mitchell Luskin was also supported in part by the Radcliffe Institute for
    Advanced Study at Harvard University. %The authors would like to thank Ilia Novikov and Ellad %   Tadmor for sharing their ideas on the mechanics of $2$D layered materials.

\bibliographystyle{siamplain}
\bibliography{Biblio}
\end{document}